\documentclass[12pt]{amsart}
\usepackage{amsmath,amssymb}
\usepackage[utf8,utf8x]{inputenc}
\usepackage[pagebackref=true,
pdftitle = {Coxeter\ and\ crystallographic\ arrangements\ are\ inductively\ free},
pdfkeywords = {arrangement\ of\ hyperplanes; Coxeter; reflection; inductively\ free; E8},
pdfsubject = {20F55; 52C35; 05E45},
pdfauthor = {Mohamed\ Barakat\ and\ Michael\ Cuntz}]{hyperref}
\usepackage{longtable}

\newtheorem{propo}{Proposition}[section]
\newtheorem{corol}[propo]{Corollary}
\newtheorem{theor}[propo]{Theorem}
\newtheorem{lemma}[propo]{Lemma}

\theoremstyle{definition}
\newtheorem{defin}[propo]{Definition}

\theoremstyle{remark}
\newtheorem{remar}[propo]{Remark}

\newcommand{\algo}[6]
{\vspace{11pt}\addtocounter{propo}{1}
\noindent{\bf Algorithm \arabic{section}.\arabic{propo}.}
{\bf #1}(#2)\\ {\it #3}.\\
{\bf Input:} #4\\
{\bf Output:} #5\\
\newcounter{#1}
\begin{list}{\textbf{\arabic{#1}.}}{\usecounter{#1}}
#6\end{list}\vspace{3pt}}

\numberwithin{equation}{section}

\newcommand{\NN }{\mathbb{N}}

\newcommand{\RR }{\mathbb{R}}

\newcommand{\ZZ }{\mathbb{Z}}

\newcommand{\Ac }{\mathcal{A}}
\newcommand{\Bc }{\mathcal{B}}
\newcommand{\Kc }{\mathcal{K}}

\newcommand{\Wc }{\mathcal{W}}
\newcommand{\Cc }{\mathcal{C}}

\DeclareMathOperator{\Aut}{Aut}

\DeclareMathOperator{\Hom}{Hom}
\DeclareMathOperator{\GL}{GL}

\newcommand{\IF }{\mathcal{IF}}
\DeclareMathOperator{\Der}{Der}
\DeclareMathOperator{\pdeg}{pdeg}

\newcommand{\frto}{\eta}

\newcommand{\rsC }{\mathcal{R}}
\newcommand{\re }{^\mathrm{re}}
\newcommand{\rer }[1]{(R\re)^{#1}}

\title[Crystallographic arrangements are inductively free]
{Coxeter and crystallographic arrangements are inductively free}

\author{M.~Barakat}
\address{Mohamed Barakat,
Fachbereich Mathematik,
Universit\"at Kaiserslau\-tern,
Postfach 3049,
D-67653 Kai\-sers\-lau\-tern, Germany}
\email{barakat@mathematik.uni-kl.de}

\author{M.~Cuntz}
\address{Michael Cuntz,
Fachbereich Mathematik,
Universit\"at Kai\-sers\-lau\-tern,
Postfach 3049,
D-67653 Kaiserslautern, Germany}
\email{cuntz@mathematik.uni-kl.de}

\begin{document}

\begin{abstract}
Using the classification of finite Weyl groupoids we prove that crystallographic arrangements, a large subclass of the class of simplicial arrangements which was recently defined, are hereditarily inductively free. In particular, all crystallographic reflection arrangements are hereditarily inductively free, among them the arrangement of type $E_8$. With little extra work we prove that also all Coxeter arrangements are inductively free.
\end{abstract}

\maketitle

\section{Introduction}

A hyperplane arrangement $\Ac$ is called free if the module of $\Ac$-derivations
$D(\Ac)$ is free.
If $\Ac$ is ``generic'', i.e.\ no additional information about the arrangement is known,
then one can prove that it is free and construct a free basis for $D(\Ac)$ by intensive
use of \textsc{Gröbner} basis techniques.

Fortunately, for certain arrangements there is a purely combinatorial method
to prove their freeness\footnote{\textsc{Terao} even conjectured that the freeness of an arrangement,
over a fixed field, only depends on the intersection lattice, and is hence a purely combinatorial notion.}
based on the Addition-Theorem \cite[Thm.\ 4.50]{OT}:
For a triple $(\Ac,\Ac',\Ac'')$ of arrangements
where $\Ac=\Ac'\cup\{H\}$, $\Ac''=\Ac^H$, the arrangement $\Ac$ is free if
$\Ac'$, $\Ac''$ are free and $\exp \Ac''\subseteq\exp \Ac'$. This theorem naturally
leads to the stronger notion of {\it inductive freeness}: The empty
arrangement is inductively free, and $\Ac$ is inductively free if there is a
triple $(\Ac,\Ac',\Ac'')$ as above where $\Ac',\Ac''$ are inductively free.
So inductive freeness implies freeness, but the converse is false \cite[4.59]{OT}.

A very important class of arrangements is the class of reflection arrangements.
There is an elegant invariant theoretic proof (see e.g.\ \cite[Thm.\ 6.60]{OT})
that reflection arrangements are free. In fact, it is conjectured
\cite[Conj.\ 6.90, Conj.\ 6.91]{OT} that reflection arrangements are inductively
free or even {\it hereditarily} inductively free.
\textsc{Orlik} and \textsc{Terao} \cite{OT2} proved that
\textsc{Coxeter} arrangements are hereditarily free.

Here we prove the inductive freeness for the following large class of simplicial arrangements:
Let $\Ac $ be a simplicial arrangement and let $R$ be a
set of nonzero covectors such that $\Ac =\{\alpha ^\perp \,|\,\alpha \in R\}$.
Assume that $\RR\alpha \cap R=\{\pm\alpha\}$ for all $\alpha\in R$.
The pair $(\Ac ,R)$ is called {\it crystallographic}, see \cite[Def.\ 2.3]{p-C10},
if for any chamber $K$ the elements of $R$ are integer linear combinations of
the covectors defining the walls of $K$. For example, all reflection arrangements
from Weyl groups\footnote{Notice that for example the arrangements of type $B_n$
and $C_n$ are not isomorphic as crystallographic arrangements because of the additional
datum $R$.} are crystallographic arrangements.
In rank greater than two the arrangements of type $H_3$ and $H_4$ are the only
\textsc{Coxeter} arrangements which are not crystallographic.

In this paper we prove that crystallographic arrangements are hereditarily
inductively free.
We treat the sporadic cases with the computer:
The algorithm is mainly based upon the fact that the roots of a finite \textsc{Weyl}
groupoid are real roots and that the finite \textsc{Weyl} groupoids are in one-to-one
correspondence with the crystallographic arrangements \cite{p-C10}.
For the non-crystallographic \textsc{Coxeter} arrangements of type $H_3$ and $H_4$ we
use a generic version of the algorithm.

The paper is organized as follows.
In Section \ref{prel} we recall all necessary definitions following \cite{OT}.
In Section \ref{gen_res} we briefly describe the classification of crystallographic
arrangements \cite{p-CH10}. In the next section we give a detailed description of the infinite series
and prove that their arrangements are inductively free.
Section \ref{wg_sporadic} treats the sporadic cases using the computer.
This is the most difficult part, in particular the arrangement of type $E_8$
requires significant extra work. All the algorithms we use to decide the
inductive freeness by searching for an ``inductive path'' produce a {\it certificate}
providing an a posteriori proof of correctness for the computed path.
In the last section we describe algorithms to decide the freeness
and compute a free basis for the module of derivations of a general central arrangement.
Finally, in the appendices we describe the computation of a free basis for the largest sporadic
crystallographic arrangement of rank $7$ and $8$ and list
the exponents of all sporadic crystallographic arrangements.

\section{Preliminaries on arrangements}\label{prel}

Let $r\in\NN$, $V:=\RR^r$.
We first recall the definition of a simplicial arrangement (compare \cite[1.2, 5.1]{OT}).
\begin{defin}\label{A_R}
Let $\Ac$ be a simplicial arrangement in $V$, i.e.~
$\Ac=\{H_1,\ldots,H_n\}$ where $H_1,\ldots,H_n$ are distinct linear hyperplanes in $V$
and every component of $V\backslash \bigcup_{H\in\Ac} H$ is an open simplicial cone.
Let $\Kc(\Ac)$ be the set of connected components of $V\backslash \bigcup_{H\in\Ac} H$;
they are called the {\it chambers} of $\Ac$.
\end{defin}

We also need the concepts of a subarrangement and restriction:

\begin{defin}[{\cite[1.12-1.14]{OT}}]
Let $(\Ac,V)$ be an arrangement. We denote $L(\Ac)$ the set of all nonempty
intersections of elements of $\Ac$.

If $\Bc\subseteq\Ac$ is a subset, then $(\Bc,V)$ is called
a {\it subarrangement}. For $X\in L(\Ac)$ define a subarrangement $\Ac_X$ of $\Ac$ by
\[ \Ac_X = \{H\in\Ac\mid X\subseteq H\}. \]
Define an arrangement $(\Ac^X,X)$ in $X$ by
\[ \Ac^X=\{X\cap H\mid H\in\Ac\backslash\Ac_X \mbox{ and } X\cap H\ne \emptyset\}.\]
We call $\Ac^X$ the {\it restriction} of $\Ac$ to $X$.

Let $H_0\in\Ac$. Let $\Ac'=\Ac\backslash\{H_0\}$ and let $\Ac''=\Ac^{H_0}$.
We call $(\Ac,\Ac',\Ac'')$ a {\it triple} of arrangements and $H_0$ the
{\it distinguished} hyperplane.
\end{defin}

Recall the module of derivations of an arrangement:

\begin{defin}[{\cite[4.1]{OT}}]
Let $(\Ac,V)$ be a real arrangement and $S=S(V^*)$ the symmetric algebra of the
dual space $V^*$ of $V$.
We choose a basis $x_1,\ldots,x_r$ for $V^*$ and identify $S$ with
$\RR[x_1,\ldots,x_r]$ via the natural isomorphism $S\cong\RR[x_1,\ldots,x_r]$.
We write $\Der(S)$ for the set of derivations of $S$ over $\RR$.
It is a free $S$-module with basis $D_1,\ldots,D_r$ where $D_i$ is the usual derivation
$\partial/\partial x_i$.

A nonzero element $\theta\in\Der(S)$ is {\it homogeneous of polynomial degree} $p$
if $\theta=\sum_{k=1}^r f_k D_k$ and $f_k\in S_p$ for $a\le k\le r$.
In this case we write $\pdeg \theta = p$.

Let $\Ac$ be an arrangement in $V$ with defining polynomial
\[ Q(\Ac) = \prod_{H\in\Ac} \alpha_H \]
where $H=\ker \alpha_H$. Define the {\it module of $\Ac$-derivations} by
\[ D(\Ac) = \{\theta\in\Der(S)\mid \theta(Q(\Ac))\in Q(\Ac)S\}. \]
\end{defin}

In this paper, we prove that certain arrangements are {\it free}:

\begin{defin}
An arrangement $\Ac$ is called a {\it free arrangement} if $D(\Ac)$ is a free
module over $S$.

If $\Ac$ is free and $\{\theta_1,\ldots,\theta_r\}$ is a homogeneous basis for
$D(\Ac)$, then $\pdeg\theta_1,\ldots,\pdeg\theta_r$ are called the {\it exponents}
of $\Ac$ and we write
\[ \exp \Ac = \{\pdeg\theta_1,\ldots,\pdeg\theta_r\}. \]
Remark that the pdegrees depend only on $\Ac$ (up to ordering).
\end{defin}

We will use the following theorem:

\begin{theor}[Addition-Deletion, {\cite[Thm.~4.51]{OT}}]\label{adddel}
Suppose $\Ac\ne\emptyset$. Let $(\Ac,\Ac',\Ac'')$ be a triple.
Any two of the following statements imply the third:
\begin{eqnarray*}
\Ac \mbox{ is free with } \exp \Ac &=& \{b_1,\ldots,b_{r-1},b_r\}, \\
\Ac' \mbox{ is free with } \exp \Ac' &=& \{b_1,\ldots,b_{r-1},b_r-1\}, \\
\Ac'' \mbox{ is free with } \exp \Ac'' &=& \{b_1,\ldots,b_{r-1}\}. \\
\end{eqnarray*}
\end{theor}

Inspired by this theorem, one defines:

\begin{defin}[{\cite[Def.~4.53]{OT}}]
The class $\IF$ of {\it inductively free} arrangements is the smallest
class of arrangements which satisfies
\begin{enumerate}
\item The empty arrangement $\Phi_\ell$ of rank $\ell$ is in $\IF$ for $\ell\ge 0$,
\item if there exists $H\in \Ac$ such that $\Ac''\in\IF$, $\Ac'\in\IF$, and
$\exp \Ac''\subset\exp \Ac'$, then $\Ac\in\IF$.
\end{enumerate}
We will say that $(\Ac_1,\ldots,\Ac_n)$ is an {\it inductive chain} of arrangements
if $\Ac_i\backslash\Ac_{i-1}=\{H_i\}$ for $i=2,\ldots,n$ and suitable $H_i$, and if
 $(\Ac_i,\Ac_{i-1},\Ac_i^{H_i})$ is a triple of inductively free arrangements for
all $i=2,\ldots,n$.
\end{defin}

We further need:

\begin{theor}[{\cite[Thm.~6.60]{OT}}]\label{reffree}
If $G$ is a finite reflection group, then its reflection arrangement
$\Ac=\Ac(G)$ is free and $\exp \Ac$ is the set of coexponents of $G$.
\end{theor}

Thus by \cite[VI.~Planche II,IV]{b-BourLie4-6}:

\begin{remar}\label{expBD}
Let $\Ac$ resp.~$\Bc$ be the reflection arrangements of type $B_r$ resp.~$D_r$.
Then
\[ \exp \Ac = \{ 1,3,\ldots,2r-3,2r-1 \}, \]
{\small
\[ \exp \Bc = \begin{cases}
\{ 1,3,\ldots,r-3,r-1,r-1,r+1,\ldots,2r-5,2r-3 \} & r \mbox{ even},\\
\{ 1,3,\ldots,r-4,r-2,r-1,r,r+2,\ldots,2r-5,2r-3 \} & r \mbox{ odd}.\\
\end{cases} \]}
\end{remar}

\section{Crystallographic arrangements}\label{gen_res}

Recall the definition of a crystallographic arrangement and
the correspondence to \textsc{Cartan} schemes:
\begin{defin}[{\cite[Def.~2.3]{p-C10}}]\label{cryarr}
Let $(\Ac,V)$ be a simplicial arrangement and $R\subseteq V$ a finite set
such that $\Ac = \{ \alpha^\perp \mid \alpha \in R\}$ and $\RR\alpha\cap R=\{\pm \alpha\}$
for all $\alpha \in R$. For a chamber $K$ of $\Ac$ set
\[ R^K_+ = R \cap \sum_{\alpha \in B^K} \RR_{\ge 0} \alpha, \]
where $B^K$ is the set of normal vectors in $R$ of the walls of $K$ pointing to the inside.
We call $(\Ac,R)$ a {\it crystallographic arrangement} if
\begin{itemize}
\item[(I)] \quad $R \subseteq \sum_{\alpha \in B^K} \ZZ \alpha$\quad for all chambers $K$.
\end{itemize}
\end{defin}

\begin{theor}[{\cite[Thm.~1.1]{p-C10}}]
There is a one-to-one correspondence between crystallographic arrangements 
and connected simply connected \textsc{Cartan} schemes for which the
real roots are a finite root system (up to equivalence on both sides).
\end{theor}

We omit the definitions of \textsc{Cartan} schemes and their root systems here because we will not
need them. It suffices to know that there is a complete classification of those
\textsc{Cartan} schemes which correspond to crystallographic arrangements (\cite[Thm.~1.1]{p-CH10}):
\begin{theor}
There are exactly three families of connected simply connected \textsc{Cartan} schemes for which
the real roots form a finite irreducible root system:
\begin{enumerate}
\item The family of \textsc{Cartan} schemes of rank two parametrized by triangulations of a
convex $n$-gon by non-intersecting diagonals.
\item For each rank $r>2$, the standard \textsc{Cartan} schemes of type $A_r$, $B_r$, $C_r$
and $D_r$, and a series of $r-1$ further \textsc{Cartan} schemes described explicitly in
Thm.\ \ref{fwg_thm:main_rank_gt8}.
\item A family consisting of $74$ further ``sporadic'' \textsc{Cartan} schemes (including those
of type $F_4$, $E_6$, $E_7$ and $E_8$).
\end{enumerate}
\end{theor}

\begin{defin}
For a finite set $\Lambda\subset \RR^r$ we will write
\[ {}^\Lambda\Ac:=\{\beta^\perp \mid \beta\in\Lambda\}. \]
Let $r,s\in\NN$ with $r\le s$.
We will say that a finite set $\Lambda\subseteq\ZZ^s$
is a \textit{root set of rank $r$} if there exists a \textsc{Cartan} scheme $\Cc$ of
rank $r$ and an injective linear map $w : \ZZ^r\rightarrow\ZZ^s$
such that $w(\rer a)=\Lambda$ for some object $a$.
\end{defin}

\section{The infinite series}\label{DD_series}

Let $r\in \NN$.
Denote $\{\alpha_1,\ldots,\alpha_r\}$ the standard basis of $\ZZ^r$.
We use the following notation: For $1\le i,j\le r$, let
\[ \frto_{i,j}:=\begin{cases}
\sum_{k=i}^j\alpha_k & i\le j \\ 0 & i>j
\end{cases}. \]

\begin{defin}
Let $Z\subseteq\{1,\ldots,r-1\}$. Let $\Xi_{r,Z}$ denote the set of roots\footnote{This set
is denoted $\Phi_{r,Z}$ in \cite{p-CH10}; we denote it $\Xi_{r,Z}$ here to avoid confusing it with the
emtpy arrangement.}
\begin{eqnarray*}
\label{set1} &\frto_{i,j-1}, & 1\le i<j\le r,\\
\label{set2} &\frto_{i,r-2}+\alpha_r, & 1\le i< r,\\
\label{set3} &\frto_{i,r}+\frto_{j,r-2}, & 1\le i<j<r,\\
\label{set4} &\frto_{j,r}+\frto_{j,r-2}, & j\in Z.
\end{eqnarray*}

Let $Y\subseteq\{1,\ldots,r-1\}$. Let $\Psi_{r,Y}$ denote the set of roots
\begin{eqnarray*}
&\frto_{i,j}, & 1\le i\le j\le r,\\
&\frto_{i,r}+\frto_{j,r-1}, & 1\le i<j<r,\\
&\frto_{j,r}+\frto_{j,r-1}, & j\in Y.
\end{eqnarray*}

Further, denote $\Psi'_{r,Y}$ the set obtained from $\Psi_{r,Y}$ by exchanging
$\alpha_{r-1}$ and $\alpha_r$.
\end{defin}

\begin{remar}
The sets $\Xi_{r,\emptyset}$ resp.\ $\Psi_{r,\{1,\ldots,r-1\}}$ are the sets of positive roots
of the \textsc{Weyl} groups of type $D_r$ resp.\ $C_r$, compare \cite[VI.\ 4.6, 4.8]{b-BourLie4-6}.
\end{remar}

The following holds (\cite[Thm.~3.21]{p-CH10}):
\begin{theor}\label{fwg_thm:main_rank_gt8}
Let $\Cc$ be a connected simply connected \textsc{Cartan} scheme of rank $r>8$ for which the real
roots form a finite irreducible root system and let
\[ \rsC_+:=\{ \rer{a}_+ \mid a \in\Cc\}. \]
Then there are two possibilities:\\
(1) The \textsc{Cartan} scheme $\Cc$ is standard ($|\rsC_+|=1$) of type $A$, $B$, $C$, $D$.\\
(2) Up to equivalence the root sets of $\Cc$ are given by
\[ \rsC_+ = \{ \Xi_{r,Z},\Psi_{r,Y},\Psi'_{r,Y} \mid Z,Y\subseteq\{1,\ldots,r-1\},\:\:|Z|=s,|Y|=s-1\} \]
for some $s\in \{1,\ldots,r-1\}$.

In particular, if $\Cc$ is not standard then it has
\[ \binom{r-1}{s-1}+\binom{r}{s} \]
different root sets and $2^{r-1}(m+r)(r-1)!$ objects.
\end{theor}

Since the sets in $\rsC_+$ are equal up to a base change, we obtain:
\begin{corol}\label{seriesutbc}
Let $\Ac$ be an irreducible crystallographic arrangement of rank $r\ge 3$ which is not sporadic.
Then up to a base change
\[ \Ac = \{ \alpha^\perp \mid \alpha \in R_+ \} \]
where $R_+$ is either a set of positive roots of type $A_r$, $B_r$, $C_r$ or $D_r$, or
$R_+=\Xi_{r,Z}$ for some $Z\subset\{1,\ldots,r-1\}$.
\end{corol}

We now treat the inductive freeness of the series:

\begin{propo}\label{prop:series_indfree}
Let $\Ac$ be an irreducible crystallographic arrangement of rank $r$ which is not sporadic.
Then $\Ac$ is an inductively free arrangement.
\end{propo}

\begin{proof}
If $\Ac$ is a reflection arrangement of type $A$, then it is inductively free
by \cite[Example 4.55]{OT}.
So assume that $\Ac$ is not of type $A$ and that $\Ac$ is not sporadic.
By Corollary \ref{seriesutbc} we may assume
$\Ac = \{ \alpha^\perp \mid \alpha\in \Xi_{r,Z} \}$
for a certain subset $Z\subseteq\{1,\ldots,r-1\}$, $Z\ne \emptyset$, or
that $\Ac$ is of type $C$. Using the base change
\[ \alpha_i \mapsto \alpha_i-\alpha_{i+1}, \quad \alpha_r \mapsto \alpha_{r-1}+\alpha_r \]
for $i=1,\ldots,r-1$, one obtains that the arrangement $\Xi_{r,Z}$ is isomorphic to
the arrangement denoted $D_r^k$ in \cite{JT} defined by
\[ x_1\cdots x_k \prod_{1\le i<j\le r} (x_i\pm x_j) \]
for $k=|Z|$.
Further, $D_r^r$ is isomorphic to the arrangements of type $B$ and $C$.
\textsc{Jambu} and \textsc{Terao} showed \cite[Example (2.6)]{JT} that $D_r^k$ are inductively free for all $r,k$.
\end{proof}

\section{Inductive freeness of the sporadic crystallographic arrangements}\label{wg_sporadic}

\subsection{Ranks three to seven}

The key to the algorithm is the fact that the root sets ``are''
sets of real roots for a \textsc{Cartan} scheme: Let $\Cc$ be a \textsc{Cartan} scheme
and $a$ an object. Then for each root $\alpha\in R_+^a$ there exists an object
$b$ and a morphism $w\in\Hom(a,b)$ such that $w(\alpha)$ is
a simple root. In particular, we get:

\begin{defin}\label{muiota} Let $\Cc$ be a \textsc{Cartan} scheme, $a$ an object and
$R^a$ the set of real roots at $a$. Then there exist maps (not unique)
\[ \mu_{R^a} : R^a \rightarrow \Hom(a,\Wc(\Cc)),
\quad \iota_{R^a} : R^a \rightarrow \{1,\ldots,r\} \]
such that for all $\alpha\in R^a$ the set $\mu_{R^a}(\alpha)(R^a)$ is a root
set and
\[ \mu_{R^a}(\alpha)(\alpha)=\alpha_{\iota_{R^a}(\alpha)} \]
is simple. Denote
$g_i(\lambda_1,\ldots,\lambda_r):=\gcd(\lambda_1,\ldots,\widehat\lambda_i,\ldots,\lambda_r)$
where ``$\widehat{\ }$'' means that we omit this element and let $\alpha\in R^a$. We set
\[ {}^\alpha R^a :=
\Big\{\frac{1}{g_{\iota_R(\alpha)}(\lambda_1,\ldots,\lambda_r)}
\sum_{\iota_R(\alpha)\ne i=1}^r \lambda_i\alpha_i \quad \mid \hspace{60pt} \]
\[ \hspace{60pt}\sum_{i=1}^r \lambda_i\alpha_i \in \mu(R^a), \quad
\sum_{\iota_R(\alpha)\ne i=1}^r \lambda_i\alpha_i\ne 0 \Big\} \]
and view ${}^\alpha R^a\subseteq\langle\alpha_1,\ldots,\widehat\alpha_{\iota(\alpha)},\ldots,\alpha_r\rangle$
as a subset of $\RR^{r-1}$.
\end{defin}

\begin{propo}
Let $\Ac$ be a crystallographic arrangement. Then $\Ac={}^R\Ac$
for some root set $R$ and
\[ {}^{{}^\alpha R}\Ac=\Ac^{\alpha^\perp}\]
for all $\alpha\in R$.
\end{propo}
\begin{proof}
We use the correspondence to \textsc{Weyl} groupoids:
Up to a base change, $\Ac=\{\alpha^\perp \mid \alpha\in R^a_+\}$ where $R^a_+$
is the set of real roots at an object $a$ of a \textsc{Cartan} scheme.
Now we use maps $\mu,\iota$ as in Def.\ \ref{muiota}.
Explicitly, $\Ac^{\alpha^\perp}$ is the arrangement given by $R^a_+$ where we erase the
$\iota(\alpha)$-th coordinate, collect the resulting vectors and possibly divide them by the
greatest common divisor of their coordinates.
\end{proof}

We also need the following proposition which is a corollary to the classification:
\begin{propo}\label{cryrescry}
Let $\Ac$ be a crystallographic arrangement and $H\in\Ac$. Then
$\Ac^H$ is crystallographic.
\end{propo}
\begin{proof}
Let $R$ be a root set with $\Ac={}^R\Ac$ and $\alpha\in R$ such that
$H=\alpha^\perp$. It suffices to check that ${}^\alpha R$ is a root set.
We use the classification.
For the $74$ sporadic crystallographic arrangements we use the computer.
The infinite series are easy to treat:
Restricting from type $A$, $B$ yields type $A$, $B$ respectively.
Restricting from type $C$ gives type $C$ except for one coordinate which
yields type $B$.

For type $D$ or the arrangements from $\Xi_{r,Z}$, $\Psi_{r,Y}$, $\Psi'_{r,Y}$ one
checks all restrictions:
In the standard ordering, restricting a root system of type $D_r$ to
$\alpha_1^\perp,\ldots,\alpha_{r-2}^\perp$ gives a root set of the form $\Xi_{r-1,Z}$
and restricting to $\alpha_{r-1}^\perp$ or $\alpha_r^\perp$ gives a root set of the form $\Psi_{r-1,Y}$.
Thus the only roots one has to consider when restricting from $\Xi_{r,Z}$, $\Psi_{r,Y}$ or
$\Psi'_{r,Y}$ are those parametrized by $Z$ resp.\ $Y$. But it is easy to check in each
case that one again obtains such a set.
\end{proof}

Remark that Prop.\ \ref{cryrescry} is only used in the algorithm and in
Cor.\ \ref{cryherfree}, where we state that the crystallographic arrangements
are hereditarily inductively free.

Let $L$ be a global variable which is a sequence of inductively free
arrangements $\Ac$ stored as pairs $(\Lambda,\exp \Ac)$ where $\Ac={}^\Lambda\Ac$.
The following algorithm treats almost all sporadic crystallographic arrangements:

\algo{IsInductivelyFree}{$r$,$R_0$,$R_1$,$e_0$,$\hat R$,$\mu$}
{Test if $R_1$ is inductively free}
{$r\in\NN$, $R_0\subseteq R_1\subseteq\NN_0^r$, $e_0=\exp {}^{R_0}\Ac$,
a root set $\hat R$ with $R_1\subseteq \hat R$, a map
$\mu=\mu_{\hat R} : \hat R \rightarrow \GL(\ZZ^r)$.}
{{\bf True} or {\bf False}, a sequence of exponents if {\bf True}.}
{
\item If $|R_0|=|R_1|$ then return ({\bf True}, $e_0$).
\item Set $F\leftarrow R_1\backslash R_0$. Sort $F$ in a ``good'' way: Call
{\bf Heuristic\-Good\-Ordering$(F)$} (see below).
\item For each $\alpha\in F$, perform steps \ref{IIFstep0} to \ref{IIFstepe}.
\item\label{IIFstep0} Set $R_2=R_0\cup\{\pm \alpha\}$. Let $R:=\mu(\alpha)(R_2)$ and
compute $R'':={}^{\mu(\alpha)(\alpha)}R$.
\item Initialize $e\leftarrow\emptyset$.
\item\label{IIFstep1} If the rank of $R''$ is $2$, then $e:=(1,|R''|-1)$, go to step \ref{IIFstepe}.
\item If $R''\in L$, then $R''$ is known to be inductively free, set $e:=\exp {}^{R''}\Ac$, go to step \ref{IIFstepe}.
\item If $|R''|=r-1$ and the rank of $R''$ is $r-1$ then set $e:=(1,\ldots,1)$ ($|e|=r$), go to step \ref{IIFstepe}.
\item Search in $L$ for a largest set $\Lambda$ with $\Lambda\subseteq R''$. If there is
no such set, then choose a linearly independent $\Lambda\in R''$ with $R''\subseteq\langle\Lambda\rangle$.
\item Let $\hat R'':={}^{\alpha}\hat R$. By Prop.\ \ref{cryrescry}, $\hat R''$ is crystallographic
and we obtain a new map $\mu_{\hat R''}$.
\item Call {\bf IsInductivelyFree}($r-1$, $\Lambda$, $R''$, $\exp{}^\Lambda\Ac$, $\hat R''$, $\mu_{\hat R''}$).
If $R''$ is inductively free, then include it with its exponents into $L$ and set $e:=\exp {}^{R''}\Ac$.
\item\label{IIFstepe} If $|e|>0$ and $e\subseteq e_0$, then $(R_0,R_2,R'')$ is a triple of arrangements
and thus by Thm.\ \ref{adddel} we know that $R_2$ is inductively free and we know $\exp {}^{R_2}\Ac$.
Call {\bf IsInductivelyFree}($r$, $R_2$, $R_1$, $\exp {}^{R_2}\Ac$, $\hat R$, $\mu$) and
return the result ({\bf True},$\exp {}^{R_1}\Ac$) if it is {\bf True}.
\item Return ({\bf False},$\emptyset$).
}

\begin{remar}
To use the above algorithm to show that all sporadic crystallographic arrangements
are inductively free, we start with the arrangements of rank three and continue up to rank
seven. After each call of the function, we store the result in $L$ as well as all the root sets
from other objects of the same \textsc{Weyl} groupoid.\\
The runtime of the algorithm strongly depends on good hash and search functions for $L$.
Further, when looking for an arrangement in $L$, we also consider arrangements with
permuted coordinates. So we need a good function that recognizes whether two matrices are
equal up to permutations of columns and rows.
\end{remar}
\begin{remar}
We also keep track of pairs $(R_0,R_1)$ for which no inductively free chain from ${}^{R_0}\Ac$ to
${}^{R_1}\Ac$ was found during the algorithm to avoid testing them again in future. As for $L$, we need
good hash functions and perform all tests up to permutations of rows and columns.
\end{remar}

The following function has proven to give good orderings (although we have to admit
that we do not know why).

\algo{HeuristicGoodOrdering}{$F$}
{Sort $F=\{\beta_1,\ldots,\beta_n\}$ in such a way that
$({}^{\{\beta_1\}}\Ac,\ldots,{}^{\{\beta_1,\ldots,\beta_n\}}\Ac)$ is hopefully
almost (up to very few transpositions) an inductive chain}
{$F\subseteq \NN_0^r$.}
{An ordering $F=\{\beta_1,\ldots,\beta_n\}$.}
{
\item $T \leftarrow F$.
\item Compute a graph $\Gamma$ having $1,\ldots,r$ as vertices and for which $(i,j)$ is
an edge if and only if $\alpha_i+\alpha_j\in T$ (this is almost the Dynkin diagram of $T$
when $T$ is a root system).
\item If possible, choose a path $i_1,\ldots,i_r$ in $\Gamma$ that passes each vertex exactly once.
\item Permute the coordinates of the elements of $T$ to $i_1,\ldots,i_r$.
\item Sort $T$ lexicographically and apply the same exchanges to $F$.
}

Algorithm 5.4 works very well for all crystallographic arrangements of
rank up to $7$ except for the largest arrangement $\Ac_{7,2}$ of rank $7$ with $91$ hyperplanes.
After several experiments one also finds a good ordering for $\Ac_{7,2}$:
\begin{figure}
\begin{minipage}{0.8\textwidth}
\begin{tiny}
\noindent
\hspace{-31pt} $F_+ := \{
(0,0,0,0,0,0,1),(0,0,0,0,0,1,0),(0,0,0,0,1,0,0),(0,0,0,1,0,0,0),(0,0,1,0,0,0,0), \vspace{-5pt} \\
(0,1,0,0,0,0,0),(1,0,0,0,0,0,0),(0,0,0,0,0,1,1),(0,0,0,0,1,0,1),(0,0,0,0,1,1,1), \vspace{-5pt} \\
(0,0,0,0,1,1,0),(0,0,0,1,0,0,1),(0,0,0,1,0,1,1),(0,0,0,1,1,1,1),(0,0,0,1,1,0,1), \vspace{-5pt} \\
(0,0,0,1,1,1,2),(0,0,1,0,0,1,0),(0,0,1,0,0,1,1),(0,0,1,0,1,1,1),(0,0,1,0,1,1,0), \vspace{-5pt} \\
(0,0,1,0,1,2,1),(0,0,1,1,1,1,1),(0,0,1,1,1,2,2),(0,0,1,1,0,1,1),(0,0,1,1,1,1,2), \vspace{-5pt} \\
(0,0,1,1,1,2,1),(0,0,1,1,2,2,2),(0,1,0,1,0,0,0),(0,1,0,1,0,0,1),(0,1,0,1,0,1,1), \vspace{-5pt} \\
(0,1,0,1,1,1,1),(0,1,0,1,1,0,1),(0,1,0,1,1,1,2),(0,1,0,2,1,1,2),(0,1,1,1,1,1,1), \vspace{-5pt} \\
(0,1,1,1,1,2,2),(0,1,1,2,1,2,2),(0,1,1,1,0,1,1),(0,1,1,1,1,2,1),(0,1,1,2,2,2,3), \vspace{-5pt} \\
(0,1,1,1,1,1,2),(0,1,1,1,2,2,2),(0,1,1,2,1,1,2),(0,1,1,2,1,2,3),(0,1,1,2,2,2,2), \vspace{-5pt} \\
(0,1,1,2,2,3,3),(0,1,2,2,2,3,3),(1,1,0,0,0,0,0),(1,1,0,1,0,0,0),(1,1,0,1,0,0,1), \vspace{-5pt} \\
(1,1,0,1,0,1,1),(1,1,0,1,1,1,1),(1,1,0,1,1,0,1),(1,1,0,1,1,1,2),(1,1,0,2,1,1,2), \vspace{-5pt} \\
(1,1,1,1,1,1,1),(1,1,1,1,1,2,2),(1,1,1,2,1,2,2),(1,1,1,1,0,1,1),(1,1,1,2,2,2,3), \vspace{-5pt} \\
(1,1,1,1,1,1,2),(1,1,1,1,1,2,1),(1,1,1,1,2,2,2),(1,1,1,2,1,1,2),(1,1,1,2,1,2,3), \vspace{-5pt} \\
(1,1,1,2,2,2,2),(1,1,1,2,2,3,3),(1,1,2,2,2,3,3),(1,2,1,2,1,2,2),(1,2,1,2,2,2,3), \vspace{-5pt} \\
(1,2,1,3,2,2,3),(1,2,1,3,2,3,4),(1,2,0,2,1,1,2),(1,2,2,3,2,3,4),(1,2,1,2,1,1,2), \vspace{-5pt} \\
(1,2,1,2,1,2,3),(1,2,1,3,1,2,3),(1,2,1,3,2,2,4),(1,2,2,3,3,4,4),(1,2,1,2,2,2,2), \vspace{-5pt} \\
(1,2,1,2,2,3,3),(1,2,1,3,2,3,3),(1,2,1,3,3,3,4),(1,2,2,2,2,3,3),(1,2,2,3,2,3,3), \vspace{-5pt} \\
(1,2,2,3,2,4,4),(1,2,2,3,3,3,4),(1,2,2,3,3,4,5),(1,2,2,4,3,4,5),(1,3,2,4,3,4,5), \vspace{-5pt} \\
(2,3,2,4,3,4,5)\}$
\end{tiny}
\end{minipage}
\caption{A set of normal vectors for $\Ac_{7,2}$.\label{fig7_91}}
\end{figure}
Let $F_+$ be defined as in Fig.\ \ref{fig7_91}
and denote $\gamma_1,\ldots,\gamma_{91}$ the elements of $F_+$ in the ordering of
Fig.\ \ref{fig7_91}. Then
\[ \{\gamma_i^\perp\mid i=1,\ldots,k\}, \quad k=1,\ldots,91 \]
is an inductive chain of arrangements. The proof is just an application of Algo.\ 5.4
(notice that all computations are now of rank six).

\subsection{\texorpdfstring{The arrangement of type $E_8$}{The arrangement of type E8}}\label{E8_arr}

In rank $8$ there is only one sporadic crystallographic arrangement, the reflection arrangement of type
$E_8$. We will denote this arrangement $\Ac_{8,1}$.
Unfortunately, Algo.\ 5.4 is not good enough for this last case (we stopped it after
a month of computation). So we have to look more closely at the structure of $\Ac_{8,1}$.
Experiments with Algo.\ 5.4 lead to the conjecture that ``{\bf HeuristicGoodOrdering}''
yields indeed an inductive chain for $\Ac_{8,1}$. So the set $R_+$ we will consider
is the one given in Fig.\ \ref{fig8_120} (in this ordering), and $\Ac_{8,1} = {}^{R_+}\Ac$.
\begin{figure}
\begin{minipage}{0.8\textwidth}
\begin{tiny}
\noindent
\hspace{-33pt} $R_+ := \{
(0,0,0,0,0,0,0,1),(0,0,0,0,0,0,1,0),(0,0,0,0,0,1,0,0),(0,0,0,0,1,0,0,0), \vspace{-5pt} \\
(0,0,0,1,0,0,0,0),(0,0,1,0,0,0,0,0),(0,1,0,0,0,0,0,0),(1,0,0,0,0,0,0,0), \vspace{-5pt} \\
(0,0,0,0,0,0,1,1),(0,0,0,0,0,1,1,0),(0,0,0,0,0,1,1,1),(0,0,0,0,1,1,0,0), \vspace{-5pt} \\
(0,0,0,0,1,1,1,0),(0,0,0,0,1,1,1,1),(0,0,0,1,1,0,0,0),(0,0,0,1,1,1,0,0), \vspace{-5pt} \\
(0,0,0,1,1,1,1,0),(0,0,0,1,1,1,1,1),(0,0,1,1,0,0,0,0),(0,0,1,1,1,0,0,0), \vspace{-5pt} \\
(0,0,1,1,1,1,0,0),(0,0,1,1,1,1,1,0),(0,0,1,1,1,1,1,1),(0,1,0,1,0,0,0,0), \vspace{-5pt} \\
(0,1,0,1,1,0,0,0),(0,1,0,1,1,1,0,0),(0,1,0,1,1,1,1,0),(0,1,0,1,1,1,1,1), \vspace{-5pt} \\
(0,1,1,1,0,0,0,0),(0,1,1,1,1,0,0,0),(0,1,1,1,1,1,0,0),(0,1,1,1,1,1,1,0), \vspace{-5pt} \\
(0,1,1,1,1,1,1,1),(0,1,1,2,1,0,0,0),(0,1,1,2,1,1,0,0),(0,1,1,2,1,1,1,0), \vspace{-5pt} \\
(0,1,1,2,1,1,1,1),(0,1,1,2,2,1,0,0),(0,1,1,2,2,1,1,0),(0,1,1,2,2,1,1,1), \vspace{-5pt} \\
(0,1,1,2,2,2,1,0),(0,1,1,2,2,2,1,1),(0,1,1,2,2,2,2,1),(1,0,1,0,0,0,0,0), \vspace{-5pt} \\
(1,0,1,1,0,0,0,0),(1,0,1,1,1,0,0,0),(1,0,1,1,1,1,0,0),(1,0,1,1,1,1,1,0), \vspace{-5pt} \\
(1,0,1,1,1,1,1,1),(1,1,1,1,0,0,0,0),(1,1,1,1,1,0,0,0),(1,1,1,1,1,1,0,0), \vspace{-5pt} \\
(1,1,1,1,1,1,1,0),(1,1,1,1,1,1,1,1),(1,1,1,2,1,0,0,0),(1,1,1,2,1,1,0,0), \vspace{-5pt} \\
(1,1,1,2,1,1,1,0),(1,1,1,2,1,1,1,1),(1,1,1,2,2,1,0,0),(1,1,1,2,2,1,1,0), \vspace{-5pt} \\
(1,1,1,2,2,1,1,1),(1,1,1,2,2,2,1,0),(1,1,1,2,2,2,1,1),(1,1,1,2,2,2,2,1), \vspace{-5pt} \\
(1,1,2,2,1,0,0,0),(1,1,2,2,1,1,0,0),(1,1,2,2,1,1,1,0),(1,1,2,2,1,1,1,1), \vspace{-5pt} \\
(1,1,2,2,2,1,0,0),(1,1,2,2,2,1,1,0),(1,1,2,2,2,1,1,1),(1,1,2,2,2,2,1,0), \vspace{-5pt} \\
(1,1,2,2,2,2,1,1),(1,1,2,2,2,2,2,1),(1,1,2,3,2,1,0,0),(1,1,2,3,2,1,1,0), \vspace{-5pt} \\
(1,1,2,3,2,1,1,1),(1,1,2,3,2,2,1,0),(1,1,2,3,2,2,1,1),(1,1,2,3,2,2,2,1), \vspace{-5pt} \\
(1,1,2,3,3,2,1,0),(1,1,2,3,3,2,1,1),(1,1,2,3,3,2,2,1),(1,1,2,3,3,3,2,1), \vspace{-5pt} \\
(1,2,2,3,2,1,0,0),(1,2,2,3,2,1,1,0),(1,2,2,3,2,1,1,1),(1,2,2,3,2,2,1,0), \vspace{-5pt} \\
(1,2,2,3,2,2,1,1),(1,2,2,3,2,2,2,1),(1,2,2,3,3,2,1,0),(1,2,2,3,3,2,1,1), \vspace{-5pt} \\
(1,2,2,3,3,2,2,1),(1,2,2,3,3,3,2,1),(1,2,2,4,3,2,1,0),(1,2,2,4,3,2,1,1), \vspace{-5pt} \\
(1,2,2,4,3,2,2,1),(1,2,2,4,3,3,2,1),(1,2,2,4,4,3,2,1),(1,2,3,4,3,2,1,0), \vspace{-5pt} \\
(1,2,3,4,3,2,1,1),(1,2,3,4,3,2,2,1),(1,2,3,4,3,3,2,1),(1,2,3,4,4,3,2,1), \vspace{-5pt} \\
(1,2,3,5,4,3,2,1),(1,3,3,5,4,3,2,1),(2,2,3,4,3,2,1,0),(2,2,3,4,3,2,1,1), \vspace{-5pt} \\
(2,2,3,4,3,2,2,1),(2,2,3,4,3,3,2,1),(2,2,3,4,4,3,2,1),(2,2,3,5,4,3,2,1), \vspace{-5pt} \\
(2,2,4,5,4,3,2,1),(2,3,3,5,4,3,2,1),(2,3,4,5,4,3,2,1),(2,3,4,6,4,3,2,1), \vspace{-5pt} \\
(2,3,4,6,5,3,2,1),(2,3,4,6,5,4,2,1),(2,3,4,6,5,4,3,1),(2,3,4,6,5,4,3,2)\}.$
\end{tiny}
\end{minipage}
\caption{A set of normal vectors for $\Ac_{8,1}$.\label{fig8_120}}
\end{figure}

If we write $R_+=\{\beta_1,\ldots,\beta_{120}\}$, then we obtain arrangements
$\Ac_1,$ $\ldots,$ $\Ac_{120}$ where $\Ac_i:=\{\beta_1^\perp,\ldots,\beta_i^\perp\}$ for $i=1,\ldots,120$.
We claim that $(\Ac_1,\ldots,\Ac_{120})$ is an inductive chain of arrangements.
To prove this, we need to check that $\Ac_2^{\beta_2^\perp},\ldots,\Ac_{120}^{\beta_{120}^\perp}$
are inductively free and that the exponents of the triples satisfy the assumptions of
Thm.\ \ref{adddel}.

First notice that each restriction $\Ac_{8,1}^H$, $H\in\Ac_{8,1}$ comes from a root set
of the sporadic finite \textsc{Weyl} groupoid $\Wc$ of rank 7 with $91$ positive roots. So whenever
we consider $\Ac^H$ for some $\Ac\subseteq\Ac_{8,1}$ and $H\in\Ac_{8,1}$, we have an action
of $\Wc$ on the corresponding ``roots'' and in particular the automorphisms of the chosen
object act as well. More precisely:

Let $F$ denote the root set of $\Wc$ for which we know an inductive chain, i.e.\ an ordering
of $F_+=\{\gamma_1,\ldots,\gamma_{91}\}$.
For each $i=1,\ldots,120$, ${}^{\mu(\beta_i)}R_+$ is some set of positive roots for $\Wc$, so
there is an automorphism $\varphi$ with
\[ \varphi({}^{\mu(\beta_i)}R_+\cup -{}^{\mu(\beta_i)}R_+) = F. \]
Now we consider $\Ac_i^{\beta_i^\perp}\subseteq {}^{\mu(\beta_i)}R_+$.
The automorphism group of the object $F$ in $\Wc$ is a subgroup of the symmetric group $S_{182}$
and it acts on $\varphi(\Ac_i^{\beta_i^\perp}\cup -\Ac_i^{\beta_i^\perp})$.
A computation yields:
\begin{lemma}
Let $O_i$ be the orbit of $\varphi(\Ac_i^{\beta_i^\perp}\cup -\Ac_i^{\beta_i^\perp})$
under the action of $\Aut(F)$ for $i=43,\ldots,120$. Then
$O_i=O_j$ whenever $|\Ac_i^{\beta_i^\perp}|=|\Ac_j^{\beta_j^\perp}|$.
\end{lemma}
\begin{remar}
The lemma is certainly also true for $i<43$ but for these cases Algo.\ 5.4 is good enough.
\end{remar}
Since these automorphims are linear maps, it suffices to check inductive freeness for one
representative of each orbit. These are arrangements with
\[ 42, 46, 49, 51, 52, 58, 60, 61, 65, 66, 68, 74, 75, 77, 80, 84, 90, 91 \]
hyperplanes.
Let $O$ be such an orbit and assume that
$\{\gamma_{i_1},\ldots,\gamma_{i_m}\}$
are the positive elements corresponding to a representative.
We choose this representative in such a way that $\max\{i_1,\ldots,i_m\}$ is minimal.
This way we ensure that the resulting ordering is very close to the prefered ordering for $F$.
Indeed, using all these techniques we can prove the inductive freeness of the reflection
arrangement of type $E_8$ in less than $5$ minutes on a usual PC with {\sf GAP}.

\subsection{Certificates for inductive freeness}

The above algorithm is quite complicated when implemented and it is very hard to
completely exclude coding errors. Therefore we use a second very short and simple
program to check that the results are correct.

\begin{defin}
Let $\Ac=\{H_1,\ldots,H_n\}$ be an inductively free hyperplane arrangement of rank at least $2$.
A {\it certificate for inductive freeness} $C_\Ac$ for $\Ac$ is
\[ C_\Ac = \begin{cases}
2, & \mbox{rank }\Ac = 2\\
((i_1,\ldots,i_n),(C_1,\ldots,C_n)), & \mbox{rank }\Ac > 2
\end{cases}, \]
where $\{i_1,\ldots,i_n\}=\{1,\ldots,n\}$, $\Ac_j:=\{H_{i_1},\ldots,H_{i_j}\}$,
$(\Ac_1,\ldots,\Ac_n)$ is an inductive chain and $C_j$ is a certificate for $\Ac_j^{H_{i_j}}$.
\end{defin}

After several modifications to the above algorithms one obtains as output the
exponents and a certificate as well. We can then check the certificate via the following:

\algo{CheckCertificate}{$\Ac,C$}
{Check whether $C$ is a certificate for $\Ac$}
{A hyperplane arrangement $\Ac$, an object $C$}
{Exponents of $\Ac$ or {\bf False}.}
{
\item Let $r$ be the rank of $\Ac$.
\item If $r=2$ then return $(1,|\Ac|-1)$.
\item Denote $\Ac=(H_1,\ldots,H_n)$ and $C=((i_1,\ldots,i_n),(C_1,\ldots,C_n))$.
\item $e\leftarrow (1,\ldots,1)$ ($r$-times).
\item For $j$ from $r+1$ to $|\Ac|$ do steps \ref{CCst1} to \ref{CCstl}.
\item \label{CCst1} $\tilde \Ac \leftarrow \{H_{i_1},\ldots,H_{i_{j-1}}\}^{H_{i_j}}$.
\item $\tilde e \leftarrow ${\bf CheckCertificate}$(\tilde A,C_{j-r})$.
\item \label{CCstl} If $\tilde e\ne${\bf False} and $\tilde e\subseteq e$ then
      $e \leftarrow \tilde e \cup (j-|\Ac|)$, else return {\bf False}.
\item Return $e$.
}

\begin{remar}
The ``$\cup$''-symbol in step \ref{CCstl} is a union of multisets, i.e.\ $e$, $\tilde e$ are
in fact sets with multiplicities.
\end{remar}
\begin{remar}
Step \ref{CCst1} is the time consuming part. Since we are only dealing with small
integers this is a function which is very easy to implement in {\sf C}.
A certificate for the arrangement $\Ac_{8,1}$ takes depending on the format between
$300$KB and $500$MB.\\
{\bf CheckCertificate} takes about $15$ minutes for $\Ac_{8,1}$ using {\sf GAP} with
a dynamic {\sf C}-module for restrictions.
It may seem surprising that it takes longer to check the certificate than to create it.
There are two reasons for this: First, Algo.\ 5.4 descends only once into each branch
it has already computed to be inductively free (remember the global variable $L$).
Secondly, in Algo.\ 5.4 we use the information on \textsc{Weyl} groupoids and morphisms and can
therefore always restrict to a simple root which amounts to erasing a
coordinate and compute gcd's. In {\bf CheckCertificate} we want to keep everything as short
as possible and in particular we transfer no information on the morphisms.
\end{remar}
Certificates for the sporadic crystallographic arrangements are available at \cite{indcert}.

\subsection{\texorpdfstring{The arrangements of type $H_3$ and $H_4$}{The arrangements of type H3 and H4}}\label{ifH3H4}

To prove that all \textsc{Coxeter} arrangements are inductively free, we still need to compute
certificates for the non-crystallographic cases. The case of rank two being trivial,
there are two arrangements left, the arrangements of type $H_3$ and $H_4$.
Fortunately these cases are of rank three and four and are small enough to be treated by
a generic version of Algo.\ 5.4 which does not use the structure of the groupoids.

\subsection{Summary}

\begin{theor}\label{thm:cryarrindfree}
Crystallographic arrangements are inductively free.
\end{theor}
\begin{proof}
This is Prop.\ \ref{prop:series_indfree} and a computation with Algo.\ 5.4.
\end{proof}

\begin{corol}\label{cryherfree}
Crystallographic arrangements are hereditarily inductively free.
\end{corol}
\begin{proof}
Let $\Ac$ be a crystallographic arrangement and $X\in L(\Ac)$. Then
$X=H_1\cap\ldots\cap H_k$ for certain hyperplanes $H_1,\ldots,H_k\in\Ac$.
Applying Prop.\ \ref{cryrescry} $k$-times, we obtain that $\Ac^X$ is
crystallographic and thus inductively free by Thm.\ \ref{thm:cryarrindfree}.
\end{proof}

\begin{corol}
\textsc{Coxeter} arrangements are inductively free.
\end{corol}
\begin{proof}
This is Thm.\ \ref{thm:cryarrindfree} and Subsection \ref{ifH3H4}.
\end{proof}

\section{\texorpdfstring{Freeness of the graded module $D(\Ac)$}{Freeness of the graded module D(A)}} \label{freeness_homalg}

In this section we describe algorithms to compute a free basis of $D(\Ac)$ for a free \emph{central} arrangement $\Ac$. These algorithms can also be used to \emph{decide} the freeness of finitely presented \emph{graded} $S$-modules and, in particular, the freeness of central arrangements.

The first two subsections describe well-known algorithms. The algorithm in the third subsection enabled us to compute a free basis for the $E_8$-arrangement given as the product of $120$ sparse matrices.

\subsection{\texorpdfstring{$D(\Ac)$ as a kernel}{D(A) as a kernel}} \label{as_a_kernel}

Expressing $D(\Ac)$ as a kernel of an $S$-module homomorphism allows the use of \textsc{Gröbner} basis techniques to compute a set of generators $D(\Ac)$ as a subset of the free module $\Der(S)$ of rank $r$, the latter being identified with $S^{1\times r}$ using the standard basis $(\theta_i := \frac{\partial}{\partial x_i}\mid i=1\ldots r)$.

By definition, $D(\Ac)$ is the kernel of the map
\[
  \psi_\Ac: \left\{ \begin{array}{ccc} \Der(S) &\to& S/\langle Q(\Ac) \rangle_S \\ \theta &\mapsto& \theta(Q(\Ac)) + \langle Q(\Ac) \rangle_S \end{array} \right. .
\]
If $Q(\Ac)$ is a complicated polynomial of large degree the \textsc{Gröbner} basis computations quickly become unfeasible. Fortunately, there exists an alternative description of $D(\Ac)$ as a kernel of some map $\phi_\Ac$ and it turns out that various \textsc{Gröbner} basis implementations scale much better when performing this kernel computation. First recall the identity \cite[Prop.~4.8]{OT}
\[
 D(\Ac) = \bigcap_{H\in\Ac} D(\alpha_H)
\]
expressing the module of $\Ac$-derivations as the intersection of $|\Ac|$ free submodules of $\Der(S)$ with
\[
D(\alpha_H):=\{\theta\in\Der(S)\mid \theta(\alpha_H)\in \langle \alpha_H \rangle_S\} = \operatorname{ker}\phi_H,
\]
where $\phi_H$ is the $S$-module map
\[
  \phi_H: \left\{ \begin{array}{ccc} \Der(S) &\to& S/\langle \alpha_H \rangle_S \\ \theta &\mapsto& \theta(\alpha_H) + \langle \alpha_H \rangle_S \end{array} \right.
\]
between $\Der(S)$ and the cyclic torsion module $S/\langle \alpha_H \rangle_S$. The intersection $D(\Ac)$ of these kernels can now be computed as the kernel of the product map $\phi_\Ac := \prod_{H\in \Ac} \phi_H$
\[
  \phi_\Ac: \Der(S) \to T_\Ac
\]
with values in the torsion $S$-module
\[
 T_\Ac := \prod_{H\in \Ac} S/\langle \alpha_H \rangle_S=S^n / \prod_{H\in \Ac} \langle \alpha_H \rangle_S.
\]

With respect to the standard generating system $(\bar{e}_j\mid j = 1 \ldots n)$ of $T_\Ac=S^n / \prod_{H\in \Ac} \langle \alpha_H \rangle_S$ we identify
\[
  T_\Ac \equiv \operatorname{coker}\left( S^{1\times n} \xrightarrow{t_\Ac} S^{1\times n} \right),
\]
where $t_\Ac$ is the $n\times n$ diagonal matrix $(\delta_{ij} \alpha_{H_j})$ with diagonal entries. The map $\phi_\Ac$ can be represented by the constant coefficients $r\times n$ matrix $f_\Ac=(\theta_i(\alpha_{H_j}))$:
\[
  \phi_\Ac: S^{1\times r} \xrightarrow{f_\Ac} \operatorname{coker}\left( S^{1\times n} \xrightarrow{t_\Ac} S^{1\times n} \right),
\]
with $\Der(S)$ identified with $S^{1\times r}$ as above.

Computing $D(\Ac)$ as the kernel of $\phi_A$ amounts to determining a generating set of solutions of the homogeneous linear system of equations over $S$
\[
  \chi f_\Ac + \eta t_\Ac = 0,
\]
or equivalently
\begin{eqnarray}\label{hom_eq}
  \left( \begin{array}{c|c} \chi & \eta \end{array} \right)
  \left( \begin{array}{c} f_\Ac \\ \hline t_\Ac \end{array} \right) = 0,
\end{eqnarray}
with $\chi\in S^{1\times r}$ and $\eta\in S^{1\times n}$. It follows that $\left\{(X_1 \mid \eta_1),\ldots,(X_q \mid \eta_q)\right\}$ is a generating set of solutions of (\ref{hom_eq}) iff $\{X_1,\ldots,X_q\}$ is a generating set of $D(\Ac)$ as a subspace of $\Der(S)\equiv S^{1\times r}$. A generating set of solutions is thus nothing but the rows of a matrix $(X\mid Y)\in S^{q\times (r+n)}$ of row syzygies of the matrix $\left( \begin{array}{c} f_\Ac \\ \hline t_\Ac \end{array} \right)$ and, as such, can be computed using a modern computer algebra system supporting \textsc{Gröbner} basis. Most such systems even provide faster procedures to compute $X$ without computing (a normal form of) $Y$ explicitly. The desired matrix $X$ is called the matrix of relative row syzygies of $f_\Ac$ modulo $t_\Ac$. Summing up: The rows $(X_1,\ldots,X_q)$ of $X$ generate $D(\Ac) \leq \Der(S)\equiv S^{1\times r}$ .

If $\Ac$ is central then all modules in this section are graded, all maps are graded of degree $0$, and the relative syzygies algorithm will produce a matrix $X$ with homogeneous rows $(X_1,\ldots,X_q)$.

\subsection{\texorpdfstring{Deciding the freeness of the graded submodule $D(\Ac)$}{Deciding the freeness of the graded submodule D(A)}}

Since $T_\Ac$ is torsion and hence of rank zero the short exact sequence
\[
  0 \to D(\Ac) \to \Der(S) \xrightarrow{\phi_\Ac} \operatorname{im}\phi_\Ac \to 0
\]
implies that $\operatorname{rank}_S D(\Ac) =\operatorname{rank}_S \Der(S) = r$, by the additivity of the rank. This means that constructing a generating set of $D(\Ac)$ with $r$ elements implies the freeness of $D(\Ac)$.

But since the number $q$ of computed generators of $D(\Ac)$ will generally exceed the rank $r$, the above argument cannot be directly applied and one needs another way to decide the freeness $D(\Ac)$.

The \textsc{Quillen-Suslin} theorem states that the freeness of an $S$-module ($S=k[x_1,\ldots,x_r]$) is equivalent to its projectiveness. But an algorithm to decide projectiveness has usually a major drawback. It does not produce a free basis.

In any case, all these algorithms take a presentation matrix of the module as their input (see \cite[§3.4]{BL} for a short survey). In our situation, where the submodule $D(\Ac)$ is only given by a set of generators $\{X_1,\ldots,X_q\}\subset S^{1\times r}$, this means that we would still need to compute a generating set of $S$-relations among the $q$ generators before entering any of these algorithms. These relations are again computable as the rows of a matrix $Z\in S^{p\times q}$ of row syzygies of the matrix $X$, and $D(\Ac)\cong\operatorname{coker}\left(S^{1 \times p} \xrightarrow{Z} S^{1\times q} \right)$.

In the large examples of interest to us the presentation matrix $Z$ usually contains huge entries. An algorithms that performs nontrivial operations on $Z$ would significantly be slowed down by the size of such entries. So it would be desirable to have an algorithm that only uses $Z$ in the cheapest possible way.

There does exist an algorithm that uses $Z$ in a very cheap way to detect obsolete rows in $X$, i.e.\  the redundant generators of $D(\Ac)$ among the rows of $X$. And fortunately, in the graded case this leads to an algorithm deciding freeness.

For the rest of the subsection let $S=\bigoplus_{i=0}^\infty S_i$ be a positively graded commutative ring with one, finitely generated as an algebra over the field\footnote{A \textsc{Noether}ian local ring $S_0$ would suffice, cf.~\cite[Exercise~20.1]{eis}} $S_0=k$ (i.e.\  $S=k[x_1,\ldots,x_m]/I$, where $I$ is a homogeneous ideal). Denote by $\mathfrak{m}=\bigoplus_{i\geq 1} S_i \triangleleft S$ the unique maximal homogeneous ideal.

\begin{propo}\label{minimal}
  Let $M$ be a graded submodule of the graded free module $S^{1\times r}$, $\mathcal{X}=\{X_1,\ldots,X_q\}$ a finite set of homogeneous generators of $M$, and $X=(X_i)_{i=1,\ldots,q}\in S^{q\times r}$ the matrix with $i$-th row $X_i$. The following conditions are equivalent:
  \begin{enumerate}
     \item $\mathcal{X}$ is minimal, i.e.\  $M$ cannot be generated by a proper subset of $\mathcal{X}$.
     \item All entries of a matrix $Z$ of row syzygies of $X$ lie in $\mathfrak{m}$.
     \item A matrix $Z$ of row syzygies of $X$ with homogeneous entries has no unit entries, i.e.\  no entries in $S_0\setminus\{0\} = k^*$.
  \end{enumerate}
  Moreover, any two minimal set of generators $\mathcal{X}$ and $\mathcal{X}'$ have the same cardinality $q$.
\end{propo}
\begin{proof}
The equivalence is a special case of the content of \cite[Section~20.1]{eis}.
\end{proof}

\begin{corol}
Let $M$ be a graded submodule of $S^{1\times r}$ of rank $r$. Then $M$ is free if and only if the cardinality of any minimal set of generators is $r$.
\end{corol}

This corollary combined with the last condition of Proposition~\ref{minimal} suggest a simple algorithm to decide freeness of the graded module $M\leq S^{1\times r}$.

\algo{GetColumnIndependentUnitPositions}{$Z$}
{Determine the position of the ``column independent'' units in the matrix $Z$}
{A matrix $Z\in S^{p\times q}$.}
{A subset $K$ of $\{1,\ldots, q\}$.}
{
\item Set $J:=\{1,\ldots, q\}$ and $K:=\emptyset$.
\item for $i\in\{1,\ldots,p\}$ and $j\in J$ do: \\
if $Z_{ij}\in k^*$, i.e.\  is a unit, then \\
redefine $K:=K\cup \{j\}$, redefine $J:=\{l\in J\mid Z_{il} = 0 \}$, and break the $j$ loop.
\item return $K$ (after finishing the $i$-loop).
}

\algo{LessGenerators}{$X$}
{Compute a minimal set of homogenous\footnote{The algorithm can be used to reduce the number of generators of a non-graded submodule given by a non-homogenous matrix. But in that case the number of rows of the output matrix has no intrinsic meaning.} generators}
{$X\in S^{q\times r}$ with homogenous rows. The rows $X_1,\ldots,X_q$ of $X$ form a set of homogeneous generators of a graded submodule $M\leq S^{1\times r}$.}
{A submatrix of $X$ with some rows eventually deleted. The set of rows of this submatrix is a minimal generating set of the graded submodule $M\leq S^{1\times r}$}
{
\item \label{Z} Compute a matrix $Z$ of row syzygies of $X$ with homogeneous entries.
\item Compute $K := \mathbf{GetColumnIndependentUnitPositions}(Z)$.
\item if $K=\emptyset$ then return $X$.
\item Define $\widetilde X$ as the matrix with rows $(X_i)_{i\in \{1,\ldots, q\}\setminus K}$.
\item return $\mathbf{LessGenerators}(\widetilde X)$.
}

The graded submodule $M\leq S^{1\times r}$ of rank $r$ is free if and only if $\mathbf{LessGenerators}(X)$ has $r$ rows, i.e.\  is a square matrix.

\subsection{\texorpdfstring{Descending chains of free submodules ending with $D(\Ac)$}{Descending chains of free submodules ending with D(A)}}

Set $S=k[x_1,\ldots,x_r]$.
Let $\beta_1,\ldots,\beta_n\in S$ be of degree $1$,
$\Ac=\{\ker \beta_1,$ $\ldots,$ $\ker \beta_n\}$ $\subset$ $k^r$ be a central arrangement with a fixed order of hyperplanes, and $\Phi_\ell =: \Ac_0 \subset \Ac_1 \subset \cdots \subset \Ac_n := \Ac$ the ascending maximal chain of central subarrangements with $\Ac_j := \{\ker \beta_1,\ldots,\ker \beta_j\}$.

The following algorithm decides the freeness of $D(\Ac_j) \leq S^{1\times r}$ for all $j \leq n$: It returns $\mathsf{fail}$ if $D(\Ac_j)$ is not free for some $j$. Otherwise it constructs a free basis $(X^{(j)}_1,\ldots,X^{(j)}_r)$ of $D(\Ac_j)$ written in the free basis $(X^{(j-1)}_1,\ldots,X^{(j-1)}_r)$ of $D(\Ac_{j-1})$ for all $1 \leq j \leq n$, starting with the standard basis $(X^{(0)}_1,\ldots,X^{(0)}_r):=(\alpha_1,\ldots,\alpha_r)$ as the free basis of $S^{1\times r}= \Der(S)=D(\Phi_\ell)=D(\Ac_0)$. In other words, the $r \times r$-matrix $X^{(j)}$, with $X^{(j)}_i$ being the $i$-th row, describes the embedding of $D(\Ac_j)$ in $D(\Ac_{j-1})$. Hence, in case all $D(\Ac_j)$ are free, the algorithm constructs the descending maximal chain of free modules
\[
  S^{1\times r}=D(\Ac_0) > \cdots > D(\Ac_n)=D(\Ac)
\]
and returns the tuple of successive embeddings $(X^{(j)})_{j=1\ldots n}$. It follows that the total embedding of $D(\Ac)$ in $S^{1\times r}$ is the product matrix $T:=\prod_{j=1}^n X^{(j)}$, the rows of which form a free basis of $D(\Ac)$ expressed in the standard basis $(\alpha_1,\ldots,\alpha_r)$ of $S^{1\times r}$.

\algo{ConstructFreeChain}{$\Ac$}
{Compute the list of successive embeddings $(X^{(j)})_{j=1\ldots n}$}
{$\Ac=\{\ker \beta_1,\ldots,\ker \beta_n\}$ with fixed order of hyperplanes.}
{A list of matrices $(X^{(j)})_{j=1\ldots n}$ describing the successive embedding of $D(\Ac_j)$ in $D(\Ac_{j-1})$}
{
\item For all $j=1\ldots n$ compute the morphism $\phi_{H_j}: S^{1\times r} \to S/\langle \beta_j \rangle$ represented by the $r\times 1$-matrix $\phi_j$ (cf.~§\ref{as_a_kernel}).
\item Compute the matrix $X$ of relative row syzygies of $\mu:=\phi_{H_1}:S^{1\times r} \to S/\langle \beta_1 \rangle$ (cf.~§\ref{as_a_kernel}).
\item If $X^{(1)}:=\mathbf{LessGenerators}(X)$ is not a square matrix then return \textsf{fail}.
\item Set $t=I_r$, the identity matrix of rank $r$.
\item For $j=2, \ldots, n$ perform steps \ref{start} to \ref{stop}:
\item \label{start} Set $t := X^{(j-1)} \cdot t$. It is the $r\times r$-matrix representing the embedding $D(\Ac_{j-1})\equiv S^{1\times r} \to \Der(S) \equiv S^{1\times r}$.
\item Set $\mu:=t \cdot \phi_j$. It is the $r\times 1$-matrix representing the morphism $D(\Ac_{j-1}) \equiv S^{1\times r}\to S/\langle \beta_j \rangle$.
\item \label{rel_syz} Compute the matrix $X$ of relative row syzygies of $\mu$ (the rows of which form a generating set of $\operatorname{ker}\mu = D(\Ac_j) < D(\Ac_{j-1})$).
\item \label{stop} If $X^{(j)}:=\mathbf{LessGenerators}(X)$ is not a square matrix then return \textsf{fail}.
\item return $(X^{(j)})_{j=1\ldots n}$.
}

\begin{remar}
This algorithm has the advantage of computing relative syzygies of morphisms represented by \emph{one-column} matrices  $\mu$ (see step \ref{rel_syz}) as opposed to the algorithms in §\ref{as_a_kernel}. The rows of the product matrix $T:=\prod_{j=1}^n X^{(j)}$ form a basis of $D(\Ac)$ and another advantage of this algorithm is that it returns the $n$ much simpler factors $(X^{(j)})_{j=1\ldots n}$ instead of $T$ itself (cf.~\cite[Theorem~4.46]{OT}).

We succeeded to compute such a descending maximal chain of free modules $S^{1\times 8} = D(\Ac_0)> \cdots > D(\Ac_{120}) = D(\Ac_{8,1})$ for the arrangement $\Ac_{8,1}$ of type $E_8$ with the roots as in Figure~\ref{fig8_120} but sorted degree reverse-lexicographically. The first $8$ roots are then the standard basis vectors $(\alpha_1,\ldots,\alpha_8)$.

In particular, the deleted arrangement $\Ac_{119}=\Ac_{8,1}'$ is free. Since $E_8$ is free\footnote{And even inductively free by Thm.\ \ref{thm:cryarrindfree}.} by Thm.\ \ref{reffree} we obtain another proof for the freeness of the restricted arrangement $\Ac_{7,2} \cong \Ac_{8,1}''$ by Thm.\ \ref{adddel}.
\end{remar}

\appendix

\section{\texorpdfstring{A free basis of $D(\Ac_{7,2})$ and of $D(\Ac_{8,1})$}{A free basis of D(A\_{7,2}) and of D(A\_{8,1})}}

Here we shortly describe the computation of a free basis of $D(\Ac_{7,2})$ and of $D(\Ac_{8,1})=D(\Ac(E_8))$. The algorithms in Section~\ref{freeness_homalg} are implemented (see \cite{indcert}) using some packages of the \texttt{homalg} project \cite{homalg-project}, written in \textsf{GAP4} \cite{GAP4}. \texttt{homalg} used \textsc{Singular} \cite{singular311} as the \textsc{Gröbner} basis engine.

The matrix $X$ of generators of $D(\Ac_{7,2})$, computed as a matrix of relative syzygies of $f_{\Ac_{7,2}}$ modulo $t_{\Ac_{7,2}}$, is a 1.9 GB $17\times 7$-matrix. The matrix $Z$ of row syzygies of $X$ is a 178 MB $10\times 17$  matrix. $\mathbf{GetColumnIndependentUnitPositions}(Z)$ returned a subset of $\{1,$ $\ldots,$ $17\}$ of cardinality $10$. $\mathbf{LessGenerators}(X)$ deleted these 10 rows from $X$ an returned a 169 MB quadratic $7\times 7$-matrix $\widetilde{X}$. It follows that the seven homogenous rows of $\widetilde{X}$ form a free basis of $D(\Ac_{7,2})$. Their degrees are $(1, 7, 11, 13, 17, 19, 23)$. The computations took 742 hours, i.e.\  around 30 days, and dropped to 6 GB of RAM at the end of the computation.

Applying \textbf{ConstructFreeChain} to the $E_8$-arrangement took only $8$ hours and $30$ minutes but needed 130 GB of RAM. These basis computations were performed on an \textsf{Opteron}-$8356$ machine with 128 GB of RAM.
The free basis is given as the rows of an $8\times8$-matrix computed as the product of $120$ sparse matrices \cite{indcert}.

Please note that proving the inductive freeness and producing the certificates for \emph{all} sporadic arrangements only took about $5$ minutes on a usual PC as mentioned at the end of §\ref{E8_arr}.

\section{Exponents of the sporadic crystallographic arrangements}

In this appendix we list the exponents of all sporadic crystallographic arrangements
$\Ac_{r,m}$ of rank $r$ and number $m$ as numbered in \cite{p-CH10}.

\begin{tiny}
\begin{longtable}{@{}ll|l||ll|l||ll|l@{}}
$r$ & $m$ & $\exp\Ac_{r,m}$ & $r$ & $m$ & $\exp\Ac_{r,m}$ & $r$ & $m$ & $\exp\Ac_{r,m}$ \\
\hline
\endhead
\hline
\endfoot
3 & 1 & 1, 4, 5 & 3 & 26 & 1, 9, 10 & 4 & 1 & 1, 4, 5, 5\\
3 & 2 & 1, 4, 5 & 3 & 27 & 1, 9, 11 & 4 & 2 & 1, 4, 5, 7\\
3 & 3 & 1, 5, 5 & 3 & 28 & 1, 9, 11 & 4 & 3 & 1, 5, 5, 7\\
3 & 4 & 1, 5, 6 & 3 & 29 & 1, 9, 11 & 4 & 4 & 1, 5, 7, 8\\
3 & 5 & 1, 5, 6 & 3 & 30 & 1, 10, 11 & 4 & 5 & 1, 5, 7, 9\\
3 & 6 & 1, 5, 7 & 3 & 31 & 1, 11, 13 & 4 & 6 & 1, 5, 7, 11\\
3 & 7 & 1, 5, 7 & 3 & 32 & 1, 11, 13 & 4 & 7 & 1, 7, 8, 9\\
3 & 8 & 1, 5, 7 & 3 & 33 & 1, 11, 13 & 4 & 8 & 1, 7, 9, 11\\
3 & 9 & 1, 5, 7 & 3 & 34 & 1, 11, 13 & 4 & 9 & 1, 7, 11, 11\\
3 & 10 & 1, 6, 7 & 3 & 35 & 1, 12, 13 & 4 & 10 & 1, 7, 11, 13\\
3 & 11 & 1, 7, 7 & 3 & 36 & 1, 12, 13 & 4 & 11 & 1, 7, 11, 13\\
3 & 12 & 1, 7, 8 & 3 & 37 & 1, 13, 13 & 5 & 1 & 1, 4, 5, 7, 8\\
3 & 13 & 1, 7, 8 & 3 & 38 & 1, 13, 13 & 5 & 2 & 1, 5, 7, 8, 9\\
3 & 14 & 1, 7, 9 & 3 & 39 & 1, 13, 13 & 5 & 3 & 1, 5, 7, 9, 11\\
3 & 15 & 1, 7, 9 & 3 & 40 & 1, 13, 14 & 5 & 4 & 1, 7, 9, 11, 13\\
3 & 16 & 1, 7, 9 & 3 & 41 & 1, 13, 14 & 5 & 5 & 1, 7, 11, 13, 14\\
3 & 17 & 1, 8, 9 & 3 & 42 & 1, 13, 14 & 5 & 6 & 1, 7, 11, 13, 17\\
3 & 18 & 1, 8, 9 & 3 & 43 & 1, 13, 15 & 6 & 1 & 1, 4, 5, 7, 8, 11\\
3 & 19 & 1, 9, 9 & 3 & 44 & 1, 13, 15 & 6 & 2 & 1, 5, 7, 9, 11, 13\\
3 & 20 & 1, 7, 11 & 3 & 45 & 1, 13, 15 & 6 & 3 & 1, 7, 11, 13, 14, 17\\
3 & 21 & 1, 9, 9 & 3 & 46 & 1, 13, 16 & 6 & 4 & 1, 7, 11, 13, 17, 19\\
3 & 22 & 1, 9, 9 & 3 & 47 & 1, 13, 17 & 7 & 1 & 1, 5, 7, 9, 11, 13, 17\\
3 & 23 & 1, 7, 11 & 3 & 48 & 1, 13, 17 & 7 & 2 & 1, 7, 11, 13, 17, 19, 23\\
3 & 24 & 1, 8, 11 & 3 & 49 & 1, 16, 17 & 8 & 1 & 1, 7, 11, 13, 17, 19, 23, 29\\
3 & 25 & 1, 9, 10 & 3 & 50 & 1, 17, 19 & & &
\end{longtable}
\end{tiny}

\providecommand{\bysame}{\leavevmode\hbox to3em{\hrulefill}\thinspace}
\providecommand{\MR}{\relax\ifhmode\unskip\space\fi MR }

\providecommand{\MRhref}[2]{%
  \href{http://www.ams.org/mathscinet-getitem?mr=#1}{#2}
}
\providecommand{\href}[2]{#2}

\end{document}